\documentclass[12pt]{amsart}
\usepackage{amssymb,amsmath,amsthm,verbatim,longtable,booktabs}
\newtheorem{theorem}{Theorem}

\newtheorem{proposition}[theorem]{Proposition}
\newtheorem{definition}[theorem]{Definition}

\theoremstyle{remark}

\newtheorem{remark}{Remark}
\newtheorem*{remarkstar}{Remark}
\newtheorem*{example}{Example}
\numberwithin{theorem}{section} \numberwithin{equation}{section}

\newcommand{\ord}{\text {\rm ord}}

\newcommand{\gz}{\Gamma_0}
\newcommand{\gzn}[1]{\gz(#1)}
\newcommand{\gzN}{\gzn{N}}

\newcommand{\C}{\mathbb{C}}

\newcommand{\Q}{\mathbb{Q}}
\newcommand{\Z}{\mathbb{Z}}

\newcommand{\Etwo}{{}_2E_1}
\newcommand{\Etwon}[1]{\Etwo(#1)}
\newcommand{\Etwol}{\Etwon{\lambda}}

\newcommand{\Ftwo}{{}_2F_1}

\newcommand{\atwo}{{}_2a_1}
\newcommand{\atwop}[1]{\atwo(p;#1)}
\newcommand{\atwopl}{\atwo(p;\lambda)}
\newcommand{\atwonl}{\atwo(n;\lambda)}

\newcommand{\bp}[1]{B(p;#1)}
\newcommand{\bpl}{B(p;\lambda)}

\newcommand{\hgss}[5]{ #1 \left( \begin{smallmatrix} #2, & #3 \\ & #4 \end{smallmatrix} \left.\right| #5 \right)}
\newcommand{\hgsl}[7]{ #1 \left( \begin{smallmatrix} #2, & #3, & \hdots, & #4 \\ & #5, & \hdots, & #6 \end{smallmatrix} \left.\right| #7 \right)}

\begin{document}

\title{Elliptic Curves, Eta-Quotients and Hypergeometric Functions}
\author{David Pathakjee, Zef RosnBrick and Eugene Yoong}
\address{Dept. of Mathematics, University of Wisconsin, Madison, Wisconsin 53706}
\email{pathakjee@wisc.edu}
\email{rosnbrick@wisc.edu}
\email{yoong@wisc.edu}
\thanks{The authors wish to thank the NSF for supporting this research, and Ken Ono and Marie Jameson for their invaluable advice.}

\begin{abstract}
The well-known fact that all elliptic curves are modular, proven by Wiles, Taylor, Breuil, Conrad and Diamond, leaves open the question whether there exists a `nice' representation of the modular form associated to each elliptic curve. Here we provide explicit representations of the modular forms associated to certain Legendre form elliptic curves $_2E_1(\lambda)$ as linear combinations of quotients of Dedekind's eta-function. We also give congruences for some of the modular forms' coefficients in terms of Gaussian hypergeometric functions.
\end{abstract}
\maketitle
\noindent

\section{Introduction and Statement of Results}
In 1996, Wiles and Taylor proved in \cite{MR1333036} that all semistable elliptic curves over $\Q$ are modular. This result was later extended by Breuil, Conrad, Diamond and Taylor in \cite{MR1839918} to all elliptic curves over $\Q$.
This correspondence allows facts about elliptic curves to be proven using modular forms and vice versa.
(See \cite{MR1216136} for more background on the theory of elliptic curves and modular forms.)
%

Let $E$ be an elliptic curve over $\Q$. If $q := e^{2 \pi i z}$, $GF(p)$ is the finite field with $p$ elements, and $N(p)$ is the number of points on $E$ over $GF(p)$, then 
the modularity theorem
 implies that there exists a corresponding weight 2 newform $f(z) = \sum_{n = 1}^{\infty}{a(n)q^n}$ such that if $p$ is a prime of good reduction, then $a(p) = 1 + p  - N(p)$.
For example, if $\eta(z)$ is Dedekind's eta-function,
\begin{align*}
	\eta(z) := q^{\frac{1}{24}} \prod_{n = 1}^{\infty}{\left(1 - q^n\right)},
\end{align*}
then the elliptic curves $y^2 = x^3 + 1$ and $y^2 = x^3 - x$ have the corresponding modular forms $\eta(6z)^4$ and $\eta(4z)^2\eta(8z)^2$, respectively \cite{MR1401749}.

It is natural to ask which elliptic curves have corresponding modular forms that are quotients of eta-functions. Martin and Ono answer this question in \cite{MR1401749} by listing all such \emph{eta-quotients}
\begin{align*}
	f(z) = \prod_\delta{\eta(\delta z)^{r_\delta}} \quad (\delta, r_\delta \in \Z)
\end{align*}
which are weight 2 newforms, and they give corresponding modular elliptic curves.
(For more on the theory of eta-quotients, see \cite[Section 1.4]{MR2020489}.)

We show, for certain values of $\lambda \in \mathbb{Q}\setminus \{0,1\}$, that the elliptic curves $\Etwol$ defined by
\begin{equation}
\Etwol :  y^2 = x(x - 1)(x - \lambda)
\end{equation}
correspond to modular forms which are linear combinations of eta-quotients.

\begin{remarkstar}
The proof of Theorem \ref{thm:etathm} will make clear how one can generate many more such examples.
\end{remarkstar}

Let
\begin{equation}
	\label{def:2a1def}
	f_\lambda(z) := \sum_{n = 1}^{\infty}{\atwonl q^n}
\end{equation}
be the weight 2 newform corresponding to the elliptic curve $\Etwol$. It will be convenient to express eta-quotients using the following notation:
\begin{equation}
	\left[\prod_\delta{\delta^{r_\delta}}\right] := \prod_\delta{\eta(\delta z)^{r_\delta}}.
\end{equation}
For example, in place of $\frac{\eta(2z)^2\eta(4z)^2\eta(5z)\eta(40z)}{\eta(z)\eta(8z)}$ we will write $[1^{-1} 2^2 4^2 5^1 8^{-1} 40^1]$.

\begin{theorem}
\label{thm:etathm}
If $\lambda \in \{\frac{27}{16}, 5, \frac{81}{49}, -\frac{7}{25}\}$, then $\Etwol$ corresponds to the modular forms given in the following table:
\begin{center}
\begin{longtable}{ c c c }
\hline
\noalign{\smallskip}
$\mathbf{\lambda}$ & \bf{Conductor $N$} & \text{\bf{Eta-quotient $f_\lambda(z)$}} \\
\noalign{\smallskip}
\hline
\noalign{\bigskip}
{\large $\frac{27}{16}$} & $33$ & $[1^2 11^2] + 3 \cdot [3^2 33^2] + 3 \cdot [1^1 3^1 11^1 33^1]$ \\[10pt]
$5$ & $40$ & $[1^{-1} 2^2  4^2  5^1  8^{-1}  40^1] + [1^1  5^{-1}  8^1  10^2  20^2  40^{-1}]$  \\[10pt]
{\large $\frac{81}{49}$} & $42$ & $2\cdot [1^{-1} 2^2 3^1 7^2 14^{-1} 42^1] - 3\cdot [3^1 6^1 21^1 42^1] $\\&& $+ [2^1 3^2 6^{-1} 7^1 21^{-1} 42^2] + [1^1 3^{-1} 6^2  14^1 21^2 42^{-1}]$ \\[10pt]
{\large $ -\frac{7}{25}$} & $70$ & $[1^{-1} 2^2 5^2 7^{-1} 10^{-1} 14^2 35^2 70^{-1}] - [1^2 2^{-1} 5^{-1} 7^2 10^2 14^{-1} 35^{-1} 70^2]$ \\
\noalign{\bigskip}
\hline
\end{longtable}
\end{center}
\end{theorem}

We show, for all $\lambda \in \Q \setminus \{0, 1\}$, that the Fourier coefficients of all $f_\lambda(z)$ satisfy an interesting hypergeometric congruence. For a prime $p$ and an integer $n$, define $\ord_p(n)$ to be the power of $p$ dividing $n$, and if $\alpha = \frac{a}{b} \in \Q$, then set $\ord_p(\alpha) = \ord_p(a) - \ord_p(b)$. We show that with this notation, the numbers $\atwopl$ satisfy the following congruences.

\begin{theorem}
\label{thm:2a1thm}
If $\lambda \in \Q \setminus\{0,1\}$ and $p = 2f + 1$ is an odd prime such that $\ord_p(\lambda(\lambda-1)) = 0$, then
\begin{align*}
\atwopl \equiv (-1)^{\frac{p+1}{2}}(p-1)\sum_{k=0}^{f}{f + k \choose k}{f \choose k}(-\lambda)^k \pmod{p}.
\end{align*}
\end{theorem}

\begin{remark}
In light of Theorem \ref{thm:etathm}, this will imply that the congruence in Theorem \ref{thm:2a1thm} holds for the coefficients of the linear combinations of eta-quotients given above.
\end{remark}

\begin{remark}
 A well-known theorem of Hasse states that for every prime $p$,
\begin{align*}
|a(p)| < 2\sqrt{p}.
\end{align*} 
Theorem \ref{thm:2a1thm} therefore determines $\atwopl$ uniquely for primes $p > 16$.
\end{remark}

\begin{example}
Consider $\lambda = \frac{27}{16}$. Then $\lambda(\lambda-1) = \frac{3^3 \cdot 11}{2^8}$ and so for $p \notin \{2, 3, 11\}$ prime we observe the congruence by inspecting the coefficients of $\Etwon{\frac{27}{16}}$ for applicable primes $p < 30$, where $\bpl$ is defined to be the right hand side of the congruence in Theorem $\ref{thm:2a1thm}$: 
\begin{center}
\begin{longtable}{ c c c }
\hline
\noalign{\smallskip}
$p$  & $\atwop{\frac{27}{16}}$ & $\bp{\frac{27}{16}}$ \\
\noalign{\smallskip}
\hline
\noalign{\smallskip}
$5$  & $-2 \equiv 3  \pmod{5}$  & $3$  \\[10pt]
$7$  & $4  \equiv 4  \pmod{7}$  & $4$  \\[10pt]
$13$ & $-2 \equiv 11 \pmod{13}$ & $11$ \\[10pt]
$17$ & $-2 \equiv 15 \pmod{17}$ & $15$ \\[10pt]
$19$ & $0  \equiv 0  \pmod{19}$ & $0$  \\[10pt]
$23$ & $8  \equiv 8  \pmod{23}$ & $8$  \\[10pt]
$29$ & $-6 \equiv 23 \pmod{29}$ & $23$ \\
\noalign{\smallskip}
\hline
\end{longtable}
\end{center}
\end{example}

\section{Elliptic Curves and Modular Forms}

In this section we prove Theorem \ref{thm:etathm}. If $E$ is an elliptic curve over $\Q$, then its conductor $N$ is a product of the primes $p$ of bad reduction for $E$, with exponents determined by the extent to which $E$ is singular over $GF(p)$. (An algorithm by Tate for computing conductors is given in \cite{MR1628193}.) Moreover, the modularity theorem implies that the modular form $f(z)$ corresponding to $E$ is an element of $S_2(\gzN)$. In particular, for an elliptic curve $\Etwol$, proving the correctness of any representation of $f_\lambda(z)$ in terms of eta-quotients amounts to checking that the given eta-quotients are elements of $S_2(\gzN)$ and checking a finite number of coefficients of their Fourier expansions against those of $f_\lambda$.

We first provide a formula for the dimension of the space of cusp forms of weight 2 and level $N$, $S_2(\gzN)$. We then show that the eta-quotients making up the linear combinations are elements of $S_2(\gzN)$ and use the dimension formula to show that equality of two elements of $S_2(\gzN)$ always depends only on some finite set of coefficients. 

The linear combinations of eta-quotients in this paper were generated by the following `algorithm': 
\begin{enumerate}
\item Given a rational number $\lambda \notin \{0,1\}$, compute the conductor $N$ of $\Etwol$. (The modular form corresponding to $\Etwol$ will be an element of $S_2(\gzN)$.)
\item Compute $\dim_\C(S_2(\gzN))$.
\item Generate eta-quotients which are elements of $S_2(\gzN)$.
\item Attempt to construct a basis for $S_2(\gzN)$ using these eta-quotients.
\end{enumerate}

Of course, once one is armed with a basis of eta-quotients for $S_2(\gzN)$, it is simple to express $f_\lambda(z)$ in terms of this basis.

\subsection{Dimension of $S_2(\gzN)$}
It will be useful to know not only that $S_2(\gzN)$ is finite-dimensional for every positive integer $N$, but also its exact dimension $d_N:=\dim_\C(S_2(\gzN))$. 

The following formula for $d_N$ is a simplification of \cite[Thm 1.34]{MR2020489}, which gives a formula for the quantity $\dim_\C(S_k(\gzN,\chi))- \dim_\C(M_{2-k}(\gzN,\chi))$, in the case where $k=2$ and $\chi = \epsilon$ is the trivial character modulo $N$.

\begin{proposition}If $N$ is a fixed positive integer and $r_p := \ord_p(N)$, then define

\begin{align*}
\lambda_p := 
\begin{cases}
p^{\frac{r_p}{2}} + p^{\frac{r_p}{2} -1} & r_p \equiv 0 \pmod{2} \\
2p^{\frac{r_p-1}{2}} & r_p \equiv 1 \pmod{2}.
\end{cases}
\end{align*}
With this notation,
\begin{equation}
\label{eqn:dimension}
d_N =
1 
+  \frac{N}{12}\prod_{p | N}(1+p^{-1})
-  \frac{1}{2}\prod_{p|N}\lambda_p
-  \frac{1}{4}\sum_{\substack{x\pmod{N}\\x^2+1\equiv0\pmod{N}}}{1}
-  \frac{1}{3}\sum_{\substack{x\pmod{N}\\x^2+x+1\equiv0\pmod{N}}}{1}.
\end{equation}

\begin{proof}
This follows from \cite[Thm 1.34]{MR2020489}, noting that the conductor of the trivial character is 1 and that $M_0(\gzN,\epsilon)$ is the space of constant functions and hence has dimension 1.
\end{proof}
\end{proposition}

\subsection{Proof of Theorem \ref{thm:etathm}}
\begin{proof}
Let $N$ be the conductor of $E = \Etwol$ and let $d_N = \dim_\C(S_2(\gzN))$ as before. Conditions under which an eta-quotient is an element of $S_2(\gzN))$ are provided in \cite[Thm 1.64 and Thm 1.65]{MR2020489}: If $f(z) = \prod_{\delta|N}\eta(\delta z)^{r_\delta}$ is an eta-quotient which vanishes at each cusp of $\gzN$, such that the pairs $(\delta,r_\delta)$ satisfy $\sum_{\delta|N}r_\delta = 4$, as well as 
\begin{align*}
\sum_{\delta \mid N}{\delta r_\delta} \equiv 0 \pmod{24}
\end{align*}
and
\begin{align*}
\sum_{\delta \mid N}{\frac{N}{\delta} r_\delta} \equiv 0 \pmod{24},
\end{align*}
then $f(z) \in S_2(\gzN)$. The order of vanishing of such an $f(z)$ at the cusp $\frac{c}{d}$ is given by \cite[Th. 1.65]{MR2020489} as
\begin{equation} 
\label{eqn:cusp}
\frac{N}{24}\sum_{\delta|N}\frac{\gcd(d,\delta)^2r_\delta}{\gcd(d,\frac{N}{d})d\delta}.
\end{equation}

It is straightforward to check that the formula above gives a positive order of vanishing for each eta-quotient at each cusp, that each eta-quotient satisfies the given congruence conditions, and that the $r_\delta$ of each eta-quotient sum to 4. These conditions guarantee that each eta-quotient appearing in the table above lies in $S_2(\gzN)$.

The eta-quotients given for $\lambda = \frac{27}{16}$ form a basis for $S_2(\gzn{33})$. Similarly, for $\lambda = 5$, the given eta-quotients along with $[2^2 10^2]$ form a basis; for $\lambda = \frac{81}{49}$ the given eta-quotients along with $[ 1^{-1} \allowbreak 2^2 \allowbreak 3^2 \allowbreak 6^{-1} \allowbreak 7^{-1} \allowbreak 14^2 \allowbreak 21^2 \allowbreak 42^{-1} ]$ form a basis; and for $\lambda = -\frac{7}{25}$ a complete basis is
\begin{eqnarray*}
&& \{ [5^2 7^2], [1^{-1} 2^2 7^2 10^1 14^{-1} 35^1 ], [10^2 14^2], [1^2 2^{-1} 5^1 7^{-1} 14^2 70^1 ], [1^2 2^{-1} 5^{-1} 7^2  10^2 14^{-1} 35^{-1} 70^2], \\
&& [1^1 5^1 7^1 35^1 ], [1^1 5^2 10^{-1} 14^1 35^{-1} 70^2], [5^1 10^1 35^1 70^1 ], [1^{-1} 2^2 5^1 7^1 35^{-1} 70^2 ] \}.
\end{eqnarray*}

%
%
%

To see this, let $g_{i,j}$ be the $j^{\text{th}}$ Fourier coefficient of the $i^\text{th}$ basis vector $g_i$ and define $t_1 < \ldots < t_{d_N}$ to be the first ascending set of indices for which the vectors $\{(g_{i,t_{j}})_{j=1}^{d_N}\}_{i=1}^{d_N}$ are linearly independent. One can find such a sequence by direct computation of the Fourier coefficients 
and inspection of the matrices $[g_{i,t_{j}}]_{i,j = 1}^{d_N}$ for various choices of small $t_1 < \ldots < t_{d_N}$.

Now let $v_i = (g_{i, t_1}, \ldots, g_{i, t_{d_N}})$ and let $b_1, \ldots, b_{d_N}$ be a basis for $S_2(\gzN)$. If we have $h_1, h_2 \in S_2(\gzN)$ with equal $t_i^\text{th}$ coefficents, then these coefficients are zero in the difference $h_1 - h_2$. But $h_1 - h_2$ can be written as a linear combination $\sum c_ib_i$ of basis elements, for constants $c_i$. Hence $\sum c_iv_i = 0$ in $\mathbb{R}^{d_N}$, so by linear independence all $c_i = 0$, and thus $h_1 - h_2 = 0$. It therefore suffices to check that the coefficients of $f_\lambda$ on $q^{t_1},\ldots,q^{t_{d_N}}$ match the coefficients that result from the linear combination of eta-quotients.
\end{proof}

\begin{remarkstar}
In practice, these computations can be done using a computer algebra system such as SAGE.
\end{remarkstar}

\begin{example}
We show that the modular form corresponding to $\Etwon{\frac{27}{16}}$ is
\begin{align*}
g(z) := [1^2 11^2] + 3 \cdot [3^2 33^2] + 3 \cdot [1^1 3^1 11^1 33^1].
\end{align*}

For convenience, let $G = \{[1^2 11^2], [3^2 33^2], [1^1 3^1 11^1 33^1]\}$ be the set of eta-quotients making up the linear combination $g(z)$. The conductor of $\Etwon{\frac{27}{16}}$ is 33 and so the corresponding modular form $f_{\frac{27}{16}}(z)$ is an element of $S_2(\gzn{33})$. To show that $g(z)$ is also an element of $S_2(\gzn{33})$, it  suffices to show that $G \subset S_2(\gzn{33})$. Take $g_i(z) \in G$. By \cite[Thm 1.64]{MR2020489}, $g_i(z)$ is a modular form of weight 2 for $\gzn{33}$. Moreover, by \cite[Thm 1.65]{MR2020489}, $g_i(z)$ vanishes at all cusps of $\gzn{33}$, and thus $g_i(z) \in S_2(\gzn{33})$.

Since $\ord_3(33) = \ord_{11}(33) = 1$, we have $\lambda_3 = \lambda_{11} = 2$ and evaluation of the dimension formula in \eqref{eqn:dimension} gives
\begin{align*}
\dim_\C(S_2(\gzn{33})) &= 1 
+ \frac{33}{12}\prod_{p | 33}(1+p^{-1})
- \frac{1}{2}\prod_{p|33}\lambda_p 
- \frac{1}{4}\sum_{\substack{x\pmod{33}\\x^2+1\equiv0\pmod{33}}} 1 
- \frac{1}{3}\sum_{\substack{x\pmod{33}\\x^2+x+1\equiv0\pmod{33}}}1 \\
& = 1 
+ \frac{33}{12}\left(1 + \frac{1}{3}\right)\left(1+\frac{1}{11}\right)
- \frac{1}{2}(\lambda_3)(\lambda_{11}) 
- \frac{1}{4}(0)
- \frac{1}{3}(0) \\
& = 3.
\end{align*}

It remains to show that $G$ is a basis for $S_2(\gzn{33})$. Any dependence relation satisfied by the elements of $G$ would imply a dependence relation among their coefficients. It thus suffices to find a set of indices $t_1 < t_2 < t_3$ such that the $3 \times 3$ matrix formed by the $t_i^\text{th}$ coefficients of these eta-quotients is nonsingular. For this particular $\lambda$, the first three coefficients suffice.

This implies that any two elements of $S_2(\gzn{33})$ which agree on the first three coefficients are equal. In fact, we observe that the first three coefficients of the modular form corresponding to $\Etwon{\frac{27}{16}}$ are the same as the first three coefficients of $g(z)$. That is, the coefficients of
\begin{align*}
g(z) = q + q^2 - q^3 - q^4 + \ldots
\end{align*}
agree with the coefficients of $f_{\frac{27}{16}}(z)$.
\end{example}

\section{Gaussian Hypergeometric Functions and Proof of Theorem \ref{thm:2a1thm}}
We recall some facts about Gaussian hypergeometric functions over finite fields of prime order and use the Gaussian hypergeometric function $\hgss{\Ftwo}{\phi}{\phi}{\epsilon}{\lambda}$ to prove Theorem \ref{thm:2a1thm}.
\subsection{Gaussian Hypergeometric Functions}
In \cite{MR879564}, Greene defined {\it Gaussian hypergeometric functions} over arbitrary finite fields and showed that they have properties analogous to those of classical hypergeometric functions.
We recall some definitions and notation from \cite{MR1407498} in the case of fields of prime order.

\begin{definition}
If $p$ is an odd prime, $GF(p)$ is the field with $p$ elements, and $A$ and $B$ are characters of $GF(p)$, define
\begin{align*}
{A \choose B} := \frac{B(-1)}{p} J(A,\bar{B}) = \frac{B(-1)}{p} \sum_{x\in GF(p)}A(x)\bar{B}(1-x).
\end{align*}
Furthermore, if $A_0,\ldots,A_n$ and $B_1,\ldots,B_n$ are characters of $GF(p)$, define the Gaussian hypergeometric series $\hgsl{{}_{n+1}F_{n}}{A_0}{A_1}{A_n}{B_1}{B_n}{x}$ by the following sum over all characters $\chi$ of $GF(p)$:
\begin{align*}
\hgsl{{}_{n+1}F_{n}}{A_0}{A_1}{A_n}{B_1}{B_n}{x} := \frac{p}{p-1}\sum_\chi {A_0\chi \choose \chi}{A_1\chi \choose B_1 \chi} \cdots {A_n \chi \choose B_n \chi}\chi(x)
\end{align*}
\end{definition}
In particular, we are concerned with the Gaussian hypergeometric series $\Ftwo(\lambda)$ defined by
\begin{align*}
\Ftwo(\lambda):= \hgss{\Ftwo}{\phi}{\phi}{\epsilon}{\lambda} = \frac{p}{p-1}\sum_\chi {\phi\chi \choose \chi}^2\chi(\lambda)
\end{align*}
where $\phi$ is the quadratic character of $GF(p)$.
It is shown in \cite{MR1407498} that if $\lambda \in \Q \setminus\{0,1\}$, then
\begin{equation}\label{eq:2F2a}
\Ftwo(\lambda) = - \frac{\phi(-1)\atwopl}{p} \end{equation}
for every odd prime $p$ such that $\ord_p(\lambda(\lambda-1)) = 0$.

In addition, define the generalized Ap\'ery number $D(n;m,l,r)$ for every $r\in \Q$ and every pair of nonnegative integers $m$ and $l$ by
\begin{align*}
D(n;m,l,r):= \sum_{k=0}^n {n+k \choose k}^m {n \choose k}^l r^{lk}.
\end{align*}
Ono also shows (ibid.) that if $p = 2f + 1$ is an odd prime and $w = l + m$, then
\begin{equation}
\label{eq:Dfmlr}
D(f;m,l,r) \equiv \left(\frac{p}{p-1}\right)^{w-1} \hgsl{{}_{w}F_{w-1}}{\phi}{\phi}{\phi}{\epsilon}{\epsilon}{(-r)^l} \pmod{p}.
\end{equation}

\begin{proof}[Proof of Theorem \ref{thm:2a1thm}]
By \eqref{eq:2F2a}  and the fact that $\phi(-1) = (-1)^{\frac{p-1}{2}}$, we have that
\begin{align*}
{\frac{p}{p-1}}\Ftwo(\lambda) = {\frac{(-1)^{\frac{p+1}{2}}\atwopl}{p-1}}.
\end{align*}
Furthermore, by \eqref{eq:Dfmlr}, letting $l = m = 1$ (and thus $w = 2$) and $r = -\lambda$, we have that
\begin{align*}
\frac{p}{p-1}\Ftwo(\lambda) \equiv D\left(f;1,1,-\lambda\right) \pmod{p}.
\end{align*}
Combining these two equations and rearranging, we get that
\begin{align*}
\atwopl \equiv (-1)^{\frac{p+1}{2}}(p-1)D\left(f;1,1,-\lambda\right) \pmod{p}.
\end{align*}
Since $D\left(f;1,1,-\lambda\right) = \sum_{k=0}^n {f + k \choose k} {f \choose k} (-\lambda)^k$, we have that
\begin{align*}
\atwopl \equiv (-1)^{\frac{p+1}{2}}(p-1)\sum_{k=0}^{f}{{f + k \choose k}{f \choose k}(-\lambda)^k} \pmod{p}.
\end{align*}
\end{proof}

\begin{remarkstar}
The binomial product ${f + k \choose k}{f \choose k}$ can be combined into the multinomial coefficient ${f + k \choose k, \ k, \ f - k}$ and so the congruence in Theorem \ref{thm:2a1thm} can also be written as
\begin{align*}
\atwopl \equiv (-1)^{\frac{p+1}{2}}(p-1)\sum_{k=0}^{f}{{f + k \choose k, k, f - k}(-\lambda)^k} \pmod{p}.
\end{align*}
\end{remarkstar}

\bibliographystyle{alpha}
\bibliography{490paper}
\end{document}